\documentclass{amsart}
\usepackage{graphicx}
\usepackage{amssymb,amscd,amsthm,amsxtra}
\usepackage{latexsym}
\usepackage{epsfig}
\usepackage{mathtools}
\usepackage{esint}
\usepackage{color}

\newtheorem{thm}{Theorem}[section]
\newtheorem{cor}[thm]{Corollary}
\newtheorem{lem}[thm]{Lemma}
\newtheorem{prop}[thm]{Proposition}
\theoremstyle{definition}
\newtheorem{defn}[thm]{Definition}
\theoremstyle{remark}

\numberwithin{equation}{section}
\newcommand{\be}{\begin{equation}}
\newcommand{\ee}{\end{equation}}

\newcommand{\R}{\mathbb R}

\newcommand{\eps}{\epsilon}

\newcommand{\p}{\partial}

\newcommand{\comment}[1]{}

\begin{document}

\title[Uniform density estimates]{Compactness estimates for minimizers of the Alt-Phillips functional of negative exponents}
\author{D. De Silva}
\address{Department of Mathematics, Barnard College, Columbia University, New York, NY 10027}
\email{\tt  desilva@math.columbia.edu}
\author{O. Savin}
\address{Department of Mathematics, Columbia University, New York, NY 10027}\email{\tt  savin@math.columbia.edu}
\begin{abstract} We investigate the rigidity of global minimizers $u \ge 0$ of the Alt-Phillips functional involving negative power potentials
$$\int_\Omega \left(|\nabla u|^2 + u^{-\gamma} \chi_{\{u>0\}}\right) \, dx, \quad \quad \gamma \in (0,2),$$
when the exponent $\gamma$ is close to the extremes of the admissible values. 

In particular we show that global minimizers in $\R^n$ are one-dimensional if $\gamma$ is close to 2 and $n \le 7$, or if $\gamma$ is close to $0$ and $n \le 4$. 

 \end{abstract}

\maketitle

 {\centering\footnotesize \textit{To David, a teacher and a friend.}\par}

\section{Introduction}

In this work we investigate minimizers of an energy functional of the type
$$ J(u, \Omega) = \int_\Omega |\nabla u|^2 + W(u) \, dx, $$
for a special class of homogenous potentials $W: \mathbb{R} \to [0, \infty)$. We are interested in the classification of global minimizers in low dimensions, a question which is intimately connected to the regularity of the free boundaries associated to these minimizers. 

This problem has been extensively studied in the literature for some particular potentials $W$. An important example is the double-well potential $$W(t)=(1-t^2)^2,$$ and the corresponding Allen-Cahn equation which appears in the theory of phase-transitions and minimal surfaces, see \cite{AlC,CH,MM, S}. 

On the other hand, free boundary problems occur when the potential $W$ is not of class $C^2$ near one of its minimum points, and minimizers can develop constant patches. Two such potentials were investigated in greater detail. The first one is the Lipschitz potential $$W(t)=t^+,$$ which corresponds to the classical obstacle problem, and we refer the reader to the book of Petrosyan, Shahgholian and Uraltseva \cite{PSU} for an introduction to this subject.
The second one is the discontinuous potential $$W(t)=\chi_{\{t>0\}},$$ with its associated Alt-Caffarelli energy, which is known as the Bernoulli free boundary problem or the two-phase problem (see \cite{AC, ACF}). We refer to the book of Caffarelli and Salsa \cite{CS} for an account of the basic free boundary theory in this setting.    
 These two examples are part of a continuous family of Alt-Phillips potentials 
 $$ W(t)= (t^+)^\beta, \quad \quad \beta \in [0,2).$$
Nonnegative minimizers $u \ge 0$ of $J$ for these power potentials, together with their free boundaries $$F(u):=\p\{u>0\},$$ were studied by Alt and Phillips in \cite{AP}. 

Recently in \cite{DS2, DS3}, we investigated properties of non-negative minimizers and their free boundaries for Alt-Phillips potentials of negative powers
$$W(t)=t^{-\gamma} \chi_{\{t>0\}}, \quad \quad \gamma \in (0,2).$$
These potentials are relevant in liquid models with large cohesive internal forces in regions of low density. 
The upper bound $\gamma <2$ is necessary for the finiteness of the energy. As $\gamma \to 2$, the energy concentrates more and more near the free boundary, and heuristically, the free boundary should minimize the surface area in the limit.

In \cite{DS2, DS3} we showed that minimizers $u \ge 0$ of the Alt-Phillips functional involving negative power potentials
 \begin{equation}\label{Efintro}
\mathcal E_\gamma(u,\Omega):=\int_\Omega \left(|\nabla u|^2 + u^{-\gamma} \chi_{\{u>0\}}\right) \, dx, \quad \quad \gamma \in (0,2), \quad \Omega \subset \R^n,
 \end{equation}
have optimal $C^\alpha$ H\"older continuity. The behaviour of $u$ near the free boundary $$F(u):=\p \{u>0\}$$
is characterized by an expansion of the type
$$ u= c_\alpha d^ \alpha + o(d^{2-\alpha}), \quad \quad \alpha:=\frac{2}{2+\gamma}\quad  \in (\frac 12, 1),$$ 
where $d$ denotes the distance to $F(u)$ and $c_\alpha d^\alpha$ represents the explicit 1D homogenous solution. Furthermore, using a monotonicity formula and dimension reduction, we showed that $F(u)$ is a hypersurface of class $C^{1,\delta}$ 
up to a closed singular set of dimension at most $n- k(\gamma)$, where $k(\gamma) \ge 3$ is the first dimension in which a non-trivial $\alpha$-homogenous minimizer exists. We also established the Gamma-convergence of a suitable multiple of the $\mathcal E_\gamma$ to the perimeter of the positivity set $Per_\Omega(\{u>0\})$ as $\gamma \to 2$.

The classification of $\alpha$-homogenous minimizers in low dimensions, i.e. finding a non-trivial lower bound for $k(\gamma)$, seems to be a difficult question. In this paper we establish such a bound in the case when $\gamma$ is sufficiently close to $2$ or to $0$. This is achieved by compactness, and the fact that the limiting problems are better understood. In particular, if $\gamma$ is close to 2 then we inherit the properties of minimal surfaces and find $k(\gamma) \ge 8$, and when $\gamma$ is close to $0$ then we inherit the properties of the Alt-Caffarelli minimizers and obtain $k(\gamma) \ge 5$. For the regularity theory of minimal surfaces, we refer the reader to Giusti's book \cite{G}. The regularity of minimizers of the Alt-Caffarelli functional:
$$\mathcal E_0 (u):= \int_\Omega (|\nabla u|^2+ \chi_{\{u>0\}})\; dx,$$ was established in dimension $n=2$ by Alt and Caffarelli \cite{AC}, in dimension $n=3$ by Caffarelli, Jerison, and Kenig \cite{CJK}, and finally in dimension $n=4$ by Jerison and Savin \cite{JS}. In view of the singular minimizing solution exhibited by De Silva and Jerison in \cite{DJ}, regularity fails in dimension $n\geq 7,$ hence $5 \leq k(0) \leq 7.$

As a consequence we have the following full regularity result for $F(u)$. 
\begin{thm}\label{TIN}
Let $u \ge 0$ be a minimizer of \eqref{Efintro}. Then $F(u)$ is of class $C^{1,\delta}$ if 
$$\mbox{either $n \le 7$ and $\gamma \in (2 - \eta,2)$, or $n \le 4$ and $\gamma \in (0, \eta)$,}$$ where $\delta,\eta>0$ are small constants.
\end{thm}

The proof of Theorem \ref{TIN} is straightforward when $\gamma$ is close to $0$, however it is much more involved when $\gamma$ is close to 2. The reason is that the estimates for the minimizers do not remain uniform as $\gamma \to 2$, even though the one-dimensional model solutions $c_\alpha [(x_n)^+]^\alpha$ do converge to a multiple of $[(x_n)^+]^{1/2}$. This can be seen from a simple example where we solve the problem in the exterior domain $\Omega=\R^n \setminus B_1$ with boundary data $u=1$ on $\p B_1$. Then the minimizer is radial, and it turns out that
$$F(u)=\p B_{1+\mu} \quad \mbox{ with $\mu \to 0$ as $\gamma \to 2$.}$$

In \cite{DS3}, we developed uniform estimates in $\gamma$ as $\gamma \to 2$, but to achieve this we had to rescale the potential term in the functional $\mathcal E_\gamma$ in a suitable way (see \eqref{Jf} in Section 2). We further established the Gamma-convergence of these rescaled functionals to the Dirichlet-perimeter functional 
 $$\mathcal F(u):=\int_\Omega |\nabla u|^2 dx + Per_{\Omega}(\{ u=0\}),$$
which was studied by Athanasopoulous, Caffarelli, Kenig, Salsa in \cite{ACKS}. The results in \cite{DS3} imply the flatness of the free boundary for global minimizers of $\mathcal E_\gamma$ up to dimension $n=7$, if $\gamma$ is close to $2$.

In this paper, we achieve the desired classification of global minimizers after establishing uniform improvement of flatness estimates, which we make precise below (see Theorem \ref{TI} in Section 2). 

Uniform estimates in other contexts were obtained for example by Caffarelli and Valdinoci for the $s$-minimal surface equation as $s \to 1$ \cite{CV}, and by Caffarelli and Silvestre for integro-differential equations as the order $\sigma \to 2$ \cite{CSi2}. 

The paper is organized as follows. In Section 2, we set notation, recall previous results, and state our main theorems. In the following section, Section 3, we provide uniform estimates for solutions to the linearized problem associated to our free boundary problem. Section 4 is devoted to the proof of the uniform improvement of flatness, which is the key tool in our strategy. In Section 5, we characterize global minimizers of $\mathcal{F}$ in low dimensions and deduce the flatness property of global minimizers of $\mathcal E_\gamma$, for $\gamma \to 2$, yielding the proof of Theorem \ref{TM} on the basis of the uniform improvement of flatness. We also prove compactness for $\gamma \to 0$, completing the proof of Theorem \ref{TM}.

\section{Main results}

We collect here our main results. We start by introducing some notation and recalling previous results.

Let $\Omega$ be a bounded domain in $\R^n$ with Lipschitz boundary. We consider the functional
 \begin{equation}\label{Ef}
\mathcal E_\gamma(u):=\int_\Omega \left(|\nabla u|^2 + u^{-\gamma} \chi_{\{u>0\}}\right) \, dx, \quad \quad \gamma \in (0,2),
 \end{equation}
 which acts on functions
$$u : \Omega \to \R, \quad \quad u \in H^1(\Omega),  \quad u \ge 0.$$
The Euler-Lagrange equation associated to the minimization problem reads
$$\Delta u = -\frac{\gamma}{2} u^{-\gamma-1},$$
and we let $u_0$ denote its one-dimensional solution
\begin{equation}\label{a}
u_0(t):= c_\alpha (t^+)^\alpha, \quad \quad  \quad \alpha = \frac{2}{2+\gamma}, \quad c_\alpha:= \left(\frac{\gamma+2}{2} \right)^\frac{2}{\gamma+2}.
\end{equation}
More precisely, we proved in \cite{DS2} that minimizers of $\mathcal E_\gamma$ are viscosity solutions to the following degenerate one-phase free boundary problem:
\begin{equation}\label{v}\begin{cases}\Delta u = -\frac \gamma 2u^{-(\gamma+1)} \quad \text{in $\{u>0\} \cap \Omega,$}\\
u(x_0 + t\nu)= c_\alpha t^{\alpha} + o(t^{2-\alpha}) \quad \text{on $F(u):=\p \{u>0\} \cap \Omega,$}
\end{cases}\end{equation}
with $t \geq 0$, $\nu$ the unit normal to $F(u)$ at $x_0$ pointing towards $\{u>0\}$, and $\alpha, c_\alpha, \gamma$ as above.

We recall the notion of viscosity solution to \eqref{v}. As usual, we say that a continuous function $u$ touches a continuous function $\phi$ by above (resp. below) at a point $x_0$ if 
$$u \geq \phi \ \text{(resp. $u\leq \phi $)}\quad \text{in a neighborhood of $x_0$}, \quad u(x_0)=\phi(x_0).$$ Typically, if the inequality is strict (except at $x_0$), we say that $u$ touches $\phi$ strictly by above (resp. below). In our context, with $\phi \geq 0$, when we say that $u$ touches $\phi$ strictly by above at $x_0$, we mean that $u \geq \phi$ in a neighborhood $B$ of $x_0$ and $u>\phi$ (except at $x_0$) in $B \cap \overline{\{\phi>0\}}$ (and similarly by below we require the inequality to be strict in a neighborhood of $x_0$ intersected $\overline{\{u>0\}}$).

We now consider the class $\mathcal C^+$ of continuous functions $\phi$ vanishing on the boundary of a ball $B:=B_R(z_0)$ and positive in $B$, such that $\phi (x)= \phi( |x-z_0|)$ in $B$ and $\phi$ is extended to be zero outside $B$. We denote by $d(x):= dist(x, \p B)$ for $x$ in $B$ and $0$ otherwise. 
Similarly we can define the class $\mathcal C^-$, with $\phi$ being zero in the ball and positive outside, and $d(x):= dist(x, \p B)$ for $x \in B^c$ and $0$ otherwise.

\begin{defn}\label{defn} We say that a non-negative continuous function $u$ satisfies \eqref{v} in the viscosity sense, if 

1) in the set where $u>0$, $u$ is $C^\infty$ and satisfies the equation in a classical sense; 

2) if $x_0 \in F(u):= \p \{u>0\} \cap \Omega$, then $u$ cannot touch $\psi \in \mathcal C^+$ (resp. $\mathcal C^-$) by above (resp. below) at $x_0$, with $$\psi(x):= c_\alpha d(x)^\alpha + \mu \, d(x)^{2-\alpha},$$ $\alpha, c_\alpha$ as in \eqref{a} and $\mu>0$ (resp $\mu<0$).
\end{defn}

Now, let $J_\gamma$ be a  rescaling of $\mathcal E_\gamma$ defined as
\begin{equation}\label{Jf}
J_\gamma (u,\Omega):=\int _\Omega |\nabla u|^2 + W_\gamma(u) \, \, dx, 
\end{equation}
where
\begin{equation}\label{1Dp00}
W_\gamma (u):= c_\gamma u^{-\gamma} \chi_{\{u>0\}}, \quad \quad  \mbox{with} \quad  c_\gamma:= \frac {1}{16} \cdot (2-\gamma)^{2}, \quad \gamma \in (0,2),
\end{equation} 
and let us introduce the Dirichlet-perimeter functional $\mathcal F$ investigated by Athanasopoulous, Caffarelli, Kenig, Salsa in \cite{ACKS}. It acts on the space of admissible pairs $(u,E)$ consisting of functions $u \ge 0$ and measurable sets $E \subset \Omega$ which have the property that $u=0$ a.e. on $E$,
$$ \mathcal A(\Omega):=\{(u,E)| \quad u \in H^1(\Omega),\quad \mbox{$E$ Caccioppoli set, $u \ge 0$ in $\Omega$, $u=0$ a.e. in $E$} \}.$$
The functional $\mathcal F$ is given by the Dirichlet-perimeter energy
$$\mathcal F_\Omega(u,E)= \int_\Omega |\nabla u|^2 dx + P_\Omega(E),$$
 where $P_{\Omega}(E)$ represents the perimeter of $E$ in $\Omega$
 \begin{align*}
 P_\Omega(E)&=\int_\Omega |\nabla \chi_E| \\
 &=\sup \,  \int_\Omega \chi_E \, div \, g \, dx \quad \mbox{with} \quad g \in C_0^\infty(\Omega), \quad |g| \le 1.  
 \end{align*}

In \cite{DS3} we established the Gamma-convergence of $J_\gamma$ to $\mathcal F$, as $\gamma \to 2$. Precisely, we proved the following theorem.

 \begin{thm}\label{TMDS3}
 Let $\Omega$ be a bounded domain with Lipschitz boundary, $\gamma_k \to 2^-$, and $u_k$ a sequence of functions with uniform bounded energies
 $$ \|u_k\|_{L^2(\Omega)} + J_{\gamma_k}(u_k,\Omega) \le M,$$
 for some $M>0$. Then, after passing to a subsequence, we can find $(u,E) \in \mathcal A(\Omega) $ such that
   $$ u_k \to u \quad \mbox{in $L^2(\Omega)$}, \quad \chi_{\{u_k>0\}} \to \chi_{E^c} \quad \mbox{in $L^1(\Omega)$}.$$
   
  Moreover, if $u_k$ are minimizers of $J_{\gamma_k}$ then the limit $(u,E)$ is a minimizer of $\mathcal F$. The convergence of $u_k$ to $u$ and respectively of the free boundaries $\p \{u_k>0\}$ to $\p E$ is uniform on compact sets  (in the Hausdorff distance sense). 
 \end{thm}
 
Here, we exploit this fact to obtain our main theorem in the subtle case when $\gamma \to 2$. The following is our main result, from which Theorem \ref{TIN} follows as discussed in the Introduction.

\begin{thm}\label{TM} Let $u$ be a global minimizer of $\mathcal E_\gamma$, and 
assume that
$$\mbox{either $n \le 7$ and $\gamma \in (2 - \eta,2)$, or $n \le 4$ and $\gamma \in (0, \eta)$,}$$ for some $\eta=\eta(n)>0$ small. Then up to rotations 
$u=u_0(x_n)$. 
\end{thm}

The key ingredient in the proof of Theorem \ref{TM} is a uniform (independent of $\gamma$) ``improvement of flatness" result for viscosity solutions. We start by giving the definition of $\eps$-flatness, $\eps>0$.
\begin{defn} \label{sre} We say that
$$ u \in \mathcal S(r,\eps) $$
if $0 \in \p \{u>0\}$ and for any ball $B_t(y) \subset B_r$ centered on the free boundary $y \in \p \{u >0\}$, $u$ is $\eps$-flat in $B_t(y)$ i.e. there exists a unit direction $\nu$ depending on $t,y,$ such that
$$u_0(x \cdot \nu - \eps t) \le u(y+x) \le u_0(x \cdot \nu + \eps t) \quad \mbox{if $|x| \le t$.}$$
\end{defn}

Our uniform improvement of flatness theorem then reads as follows.

\begin{thm}\label{TI}Let $u$ be a viscosity solution to \eqref{v} in $B_1$.
There exists $\eps_0(n)>0$ such that if $0<\eps \le \eps_0$ then
$$ u \in \mathcal S(1,\eps) \quad \Longrightarrow \quad u \in \mathcal S(\rho, \frac \eps 2 ),$$
for some $\rho(n)>0$.
\end{thm}
The proof of this theorem relies on a compactness argument which linearizes the problem into a fixed-boundary degenerate linear problem with a ``Neumann" boundary condition, whose property we analyze in the next section.

\section{Uniform estimates for the linearized equation}

Here, we discuss the linearized problem associated to our free boundary problem. We refer to some of our previous results in \cite{DS1}, where we studied the regularity of the free boundary for a class of degenerate problems.

For $s \in (-1,1)$ we consider solutions of 
\begin{equation}\label{LE1}
\Delta v + s \frac{v_n}{x_n} =0 \quad \mbox{in $B_1^+$,}  
\end{equation}
which satisfy the ``Neumann" boundary condition
\begin{equation}\label{LE2}
\p_{x_n^{1-s}} v =0 \quad \mbox{on $\{x_n=0\}$,}
\end{equation}
in the viscosity sense. This means that $v$ cannot be touched by below (above) locally by the family of comparison functions
$$p(x') + t x_n^{1-s} \quad \mbox{with $t>0$ ( or $t<0$) and $p(x')$ quadratic},$$
at points on $\{x_n=0\}$. There is a unique solution to the Dirichlet problem which assigns continuous data on $\p B_1 \cap \{x_n \ge 0\}$, and in the case of sufficiently regular data the solution is the minimizer of the energy
$$ \int_{B_1^+} |\nabla v|^2 \, x_n^{s} dx_n.$$ 

Problem \eqref{LE1} in ${\R^n}^+$ is the extension problem of Caffarelli-Silvestre \cite{CSi1}, with the Dirichlet to Neumann operator representing $\Delta^{\frac{1-s}{2}}v$ on $\{x_n=0\}$. This is well-known problem, yet here we are interested in its local version, which appears in connection with a variety of degenerate free boundary problems. In \cite{DS1}, we studied a version of \eqref{LE1} for more general linear operators that do not necessarily have a variational structure. We make some remarks on the range of $s$ in equation \eqref{LE1}. When $s \in (-1,1)$, both the Dirichlet and the Neumann boundary conditions can be imposed on $x_n=0$. However, when $s \geq 1$, only the Neumann condition can be imposed and it simply requires the function to be bounded. When $s \leq -1$, the Dirichlet condition is meaningful, but the Neumann condition in the sense defined above cannot be imposed.

Now, we start by obtaining uniform (independent on $s$) $C^{1,\alpha}$ estimates for solutions to this Neumann problem.

\begin{prop}[Uniform $C^{1,\alpha}$ estimates]\label{UE}
Let $s \in (-1,0]$, and let $v$ be a solution to the Neumann problem \eqref{LE1}-\eqref{LE2}. Then 
\begin{equation}\label{v-v0}
|v- v(0)-a' \cdot x'| \le C |x|^{1+\alpha}\|v\|_{L^\infty}, \quad \quad \mbox{for some $a'\in \R^{n-1}$,}
\end{equation}
with $C,\alpha >0$ depending only on $n$ (but not on $s$).

\end{prop}

\begin{proof}

In \cite{DS1} we obtained the estimate above for constants $C$ and $\alpha$ that depend on $s$. The proof in \cite{DS1} does not fully apply here as the constants in the main Harnack estimate (Lemma 7.4) deteriorate as $s \to -1$.

We recall briefly the key steps from \cite{DS1} which imply the Proposition in the case when we restrict $s$ to a compact interval of $(-1, 0]$, and then we explain how to modify the argument when $s$ is close to $-1$. 

In Proposition 7.5 of \cite{DS1} we showed that the difference of two viscosity solutions is a viscosity solution. Then the conclusion follows by iterating a Harnack estimate of the type
\begin{equation}\label{Hi0}
osc_{B^+_{\rho}} \, v \le (1-\eta)   \,  osc_{B^+_{1}} \, v, \quad \quad \rho,\eta>0,
\end{equation}
for discrete differences in the $x'$ direction. The Harnack estimate \eqref{Hi0} was achieved by writing the interior estimate in a ball $B_{1/4}(e_n/2)$ and then extending it to the flat boundary with the aid of explicit barriers of the type
$$ -|x'|^2 + \frac{1}{1+s} x_n^2 \pm \eps x_n^{1-s}.$$
These barriers degenerate as $s \to -1$, and so do the constants $\eta$ and $\rho$.

Below we show by a different argument that the estimate \eqref{Hi0} holds with constants $\rho$ and $\eta$ that depend only on the dimension $n$ when $s$ is near to $-1$. 

Let us assume 
$$0 \le v \le 1 \quad \mbox{ in} \quad B_1^+,$$ 
and we claim that 
\begin{equation}\label{1000}
\mbox{ either} \quad v \ge \eta \quad \mbox{ or} \quad v \le 1-\eta \quad \mbox{in} \quad  B_ \rho^+,
\end{equation}
 for small constants $\eta$ and $\rho$ to be specified below.

The limit 
$$\lim_{x_n \to 0^+} v_n x_n^{s} =0$$
exists in the classical sense, hence $v$ is also a weak solution i.e. 
$$ \int_{B_1^+} \nabla v \cdot \nabla w \, \, x_n^{s} d x =0,$$
for any $C^1$ function $w$ which vanishes near $\p B_1$.
By taking $w=\varphi^2 v$ with $\varphi$ a cutoff function which is 1 in $B_{1/2}$ and $0$ near $\p B_1$, we find the Caccioppoli inequality
$$ \int_{B_{1/2}^+} |\nabla v|^2 x_n^{s} dx \le C \int_{B_{1}^+} v^2 x_n^{s} dx $$
with $C$ independent of $s$. We iterate this $m=m(n)$ times by taking derivatives in the $x'$ direction, and obtain
\begin{equation}\label{1000.5}
\int_{B_{2^{-m}}^+} |D^m_{x'} v|^2 x_n^{s} dx \le C(m) \int_{B_{1}^+} v^2 x_n^{s} dx.
\end{equation}
Let $a>0$ be given by 
$$a^{1+s}:= \frac 12,$$
so that
$$\int_0^a x_n^{s} d x_n= \frac 12 \int_0^1 x_n^{s}dx_n.$$ Notice that $a \to 0$ as $s \to -1$, which means that the weight $x_n^{s}$ concentrates its mass near $x_n=0$ when $s$ is close to $-1$. Since $v \le 1$, \eqref{1000.5} implies that there exists $ t \in (0,a]$ such that
$$\int_{B_{2^{-m}}^+ \cap \{x_n=t\}} |D^m_{x'} v|^2 dx' \le C(m).$$
We choose $m$ sufficiently large so that the Sobolev embedding gives
$$| \nabla_{x'} v| \le C \quad \mbox {on} \quad B_{2^{-{(m+1)}}}^+ \cap \{x_n=t\}.$$
Assume that at $t e_n$, the value of $v$ is closer to 1 than to 0, i.e 
$$ v(t e_n) \ge \frac 1 2.$$
Then, the Lipschitz bound in $x'$ implies
\begin{equation}\label{1001}
v(x',t) \ge c_0 - C_0 |x'|^2,
\end{equation}
for constants $c_0$ small, and $C_0$ large that depend only on $n$. 

In the cylinder 
$$ \mathcal C_1:=B'_{1/2}\times [0,t]$$
we use as lower barrier
$$q_1(x):= \frac {c_0}{ 2} + C_0 \left (- |x'|^2 +  \frac{1}{1
+s} x_n^2 +  x_n^{1-s} \right).  $$
Notice that $q_1$ satisfies the equation \eqref{LE1} and is a strict subsolution for the Neumann condition on $\{x_n=0\}$. Moreover, on $x_n=t$, we use that $t \le a$ to find
$$\frac{1}{1+s} x_n^2 + x_n^{1-s} \le \frac{1}{1+s} a^2 + a^{1-s} \to 0 \quad \mbox{as $s \to -1$,}$$
hence $q_1 \le v$ by \eqref{1001}. The inequality above shows that $q_1 \le 0 \le v$ on the lateral boundary of $\mathcal C_1$ and by the maximum principle we find $$ v \ge q_1 \quad \mbox{in} \quad \mathcal C_1.$$
In the cylinder 
$$ \mathcal C_2:=B'_{1/2}\times [t,\frac 12]$$
we use as lower barrier
$$q_2(x):=  c_0  + C_0 \left (- |x'|^2 +  \frac{1}{1+s} (x_n^2 -  x_n^{1-s}) \right).  $$
Notice that as $ s \to -1$, 
\begin{equation}\label{1002}
\frac{1}{1+s} (x_n^2 -  x_n^{1-s}) \to x_n^2 \, \log x_n,
\end{equation}
uniformly in $[0,1]$, hence
$q_2 \le 0 \le v$ on $\p \mathcal C_2 \setminus \{ x_n =t\}$. On the remaining part of the boundary $q_2 \le v$ by \eqref{1001}. In conclusion $$v \ge q_2 \quad \mbox{ in} \quad \mathcal C_2.$$
The claim \eqref{1000} is proved since $q_1$ and $q_2$ are both greater than $\eta:=c_0/3$ in a sufficiently small ball $B_\rho^+$ of fixed radius. 

We can iterate the claim \eqref{1000}  and obtain that solutions to \eqref{LE1}-\eqref{LE2} which are normalized so that
$$\|v\|_{L^\infty(B_1^+)}=1,$$
satisfy 
$$ |v(x)-v(0)| \le C |x|^\beta,$$
for some fixed $\beta>0$, and for all $s$ sufficiently close to $-1$. 
Using scaling and interior estimates, we find
$$
\|v\|_{C^\beta(B^+_{1/2})} \le C.$$
This estimate can be iterated for discrete differences in the $x'$ direction, and we obtain
\begin{equation}\label{1003}
\|D_{x'}^m v\|_{C^\beta(B^+_{1/2})} \le C(m).
\end{equation}
After subtracting a linear function in the $x'$ variable we may assume further that
$$ |v(x',0)| \le C |x'|^2,$$
for some $C$ large. Now we bound $v$ in the remaining $x_n$ direction by trapping it between the barriers (see \eqref{1002})
$$\pm C \left (|x'|^2 -  \frac{1}{1+s} (x_n^2 -  x_n^{1-s}) \right).$$
We obtain the desired conclusion 
$$|v| \le C |x|^{3/2},$$
since
$$\frac{1}{1+s} |x_n^2 -  x_n^{1-s}| \le |x_n|^{1-s} |\log \, x_n|.$$

\end{proof}

Next we show that the limiting equation to \eqref{LE1}-\eqref{LE2} as $s \to -1$ is
\begin{equation}\label{1004}
\Delta v - \frac{v_n}{x_n} =0 \quad \mbox{in} \quad B_1^+, 
\end{equation} 
with boundary data on $\{x_n=0\}$ which is harmonic 
\begin{equation}\label{1005}
\Delta_{x'} v(x',0) =0.
\end{equation} 
\begin{lem}\label{L08}
Let $v_m$ be a sequence of solutions to \eqref{LE1}-\eqref{LE2} for $s_m \to -1$, with $\|v_m\|_{L^\infty(B_1^+)}\le 1$. Then there exists a subsequence 
which converges uniformly on compact sets to a solution of \eqref{1004}-\eqref{1005}.

Conversely, any continuous solution of \eqref{1004}-\eqref{1005} is the limit of a sequence of solutions $v_m$ to \eqref{LE1}-\eqref{LE2} with $s_m \to -1$.
\end{lem}

\begin{cor}\label{C09}
A continuous solution to \eqref{1004}-\eqref{1005} satisfies the $C^{1,\alpha}$ estimate \eqref{v-v0} in Proposition $\ref{UE}.$
\end{cor}

\

\noindent\textit{Proof of Lemma $\ref{L08}$.}
 By \eqref{1003} we find that, after passing to a subsequence, we can extract a subsequence $s_m \to -1$ such that 
 $$v_{m} \to v \quad \mbox{ uniformly in} \quad  B_1\cap \{x_n\ge 0\},$$
 with $v$ a H\"older continuous function. Clearly $v$ satisfies \eqref{1004}. 
 
 We show that $v$ satisfies \eqref{1005} in the viscosity sense. Assume by contradiction that on $x_n=0$ we can touch $v$ strictly by below in $B_\delta'$, say at $0$, with a quadratic polynomial $p(x')$ with $$\Delta_{x'} p =t >0.$$ Then, we choose $M$ sufficiently large such that 
 $$ v > p(x') - M x_n^2 +  \frac {t}{4} x_n^2 |\log \, x_n|  \quad \quad \mbox{on} \quad \p B_\delta \cap \{ x_n \ge 0\}.$$
 The uniform convergence of the $v_m$'s to $v$ and \eqref{1002} imply that $v_m$ can be touched by below in $B_\delta \cap \{ x_n \ge 0\}$ by 
 $$ p(x') - M x_n^2 + \frac{t}{4(1+s_m)} (x_n^{1-s_m} - x_n^2) + const.$$
 It is straightforward to check that this function is a strict subsolution to \eqref{LE1}-\eqref{LE2} for all large $m$, and we reached a contradiction.
 
 For the second part, let $\phi$ be a continuous function on $\p B_1$ and let $v_s$ be the solution to the Neumann problem \eqref{LE1}-\eqref{LE2} with prescribed data $\phi$ on $\p B_1$. Then, it suffices to show that as $s \to -1$, $v_s$ converges uniformly in $\overline{B_1^+}$ to the solution of \eqref{1004}-\eqref{1005} with boundary data $\phi$.
 
 The convergence in the interior region $B_1 \cap \{x_n \ge 0\}$ was proved above. On the boundary $\p B_1 \cap \{x_n \ge 0\}$ this follows from standard barrier arguments. Indeed, at the points on $\p B_1 \cap \{x_n=0\}$ the linear functions in $x'$ variables act as barriers for all $v_s$, while at the remaining points on $\p B_1 \cap \{x_n>0\}$ the limiting equation is nondegenerate.
\qed

\section{Proof of Theorem \ref{TI}}

In this section we provide the proof of our main uniform improvement of flatness Theorem \ref{TI}.

After a change of variables we rewrite the equation in \eqref{v} in the form
$$\Delta w = \frac{h_s(\nabla w)}{w}.$$
Precisely, we denote 
\begin{equation}\label{change} w:= c_\alpha^ {- \frac 1 \alpha} u^\frac 1 \alpha,\end{equation}
so that $u = c_\alpha w^\alpha$. The equation for $w$ is
$$c_\alpha \alpha w ^{\alpha -2} \left( w\Delta w + (\alpha -1)  |\nabla w|^2 \right) = - \frac \gamma 2 c_\alpha^{-(\gamma+1)} w^{-\alpha(\gamma+1)},$$
and using \eqref{a} we find
$$ \Delta w = (1-\alpha) \frac{|\nabla w|^2 -1}{w}.$$
We write this as
\begin{equation}\label{weq}
 \Delta w =  :\frac{h_s(\nabla w)}{w}=s \frac{h(\nabla w)}{w}
 \end{equation}
with
 $$h(\nabla w)= \frac 12 (1-|\nabla w|^2) , \quad \quad  \quad s :=2(\alpha-1) \in (-1,0).$$
Here $h$ is the radial quadratic function which vanishes on $\p B_1$, it is negative in $B_1$ and positive outside $B_1$ and 
\begin{equation}\label{nablah}
 \nabla h( \omega) = - \, \omega, \quad \mbox{if} \quad \omega \in \p B_1. 
 \end{equation}
Notice that \eqref{weq} remains invariant under the rescaling
$$\ \tilde w(x)=\frac {w(rx)} {r}.$$
In view of the viscosity definition for $u$, we find that $w$ satisfies \eqref{weq} with the following free boundary condition on $\p \{w>0\}$:

\begin{defn}\label{defnw} We say that $w: \Omega \to \mathbb R^+$ satisfies \eqref{weq} in the viscosity sense, if $w$ is $C^\infty$ and satisfies the equation in the set $\{w>0\} \cap \Omega$ and, if $x_0 \in F(w):= \p \{w>0\} \cap \Omega$, then $w$ cannot touch $\psi \in \mathcal C^+$ (resp. $\mathcal C^-$) by above (resp. below) at $x_0$, with $$\psi(x):= d(x) + \mu \,  d(x)^{1-s},$$ $\alpha$ as in \eqref{a} and $\mu>0$ (resp. $\mu<0$).
\end{defn}
Similarly, we keep the same notation as in Definition \ref{sre} for the corresponding solutions $w$ to \eqref{weq}. Precisely, 
\begin{defn} \label{srew}We say that 
$$ w \in \mathcal S(r,\eps) $$
if $0 \in \p \{ w>0\}$ and for any ball $B_t(y) \subset B_r$ centered on the free boundary $y \in \p \{w >0\}$, $w$ is $\eps$-flat in $B_t(y)$ i.e. there exists a unit direction $\nu$ depending on $t, y,$ such that
$$(x \cdot \nu - \eps t)^+ \le w(y+x) \le (x \cdot \nu + \eps t)^+ \quad \mbox{if $|x| \le t$.}$$

\end{defn}

We state the main result of this section from which Theorem \ref{TI} easily follows.

\begin{prop}\label{P1}
Assume that $w$ is a viscosity solution of \eqref{weq} in $B_1$, and
$$ w \in \mathcal S(1,\eps) ,$$
for some $0<\eps \le \eps_0(n)$ small. Then, there exists $\rho>0$ depending on $n$ such that
$$ (x \cdot \nu - \frac \eps 2 \rho)^+ \le w \le (x \cdot \nu + \frac \eps 2 \rho)^+ \quad \mbox{in $B_\rho$},$$
for some unit direction $\nu$, $|\nu|=1$.
\end{prop}

We remark that the Proposition \ref{P1} can be applied for all balls $B_t(y) \subset B_1$ centered on the free boundary. Then the conclusion implies that
$$ w \in \mathcal S(\rho, \frac \eps 2),$$
which is precisely the statement of Theorem \ref{TI}.

In Proposition 6.1 of \cite{DS1} we proved Proposition \ref{P1} for constants $\eps_0$ and $\rho$ that depend on $s$. Below we show that this dependence can be dropped.

After a rotation we may assume that 
\begin{equation}\label{2000}
(x_n-\eps)^+ \le w \le (x_n+ \eps)^+ \quad \mbox{in $B_1$.}
\end{equation}
As in \cite{DS1}, we consider the rescaled function
\begin{equation}\label{tidw}
 \tilde w = \frac{w-x_n}{\eps} ,
 \end{equation}
and show that is well approximated by a viscosity solution of the linearized equation
\begin{equation}\label{LiE}
\begin{cases}
\Delta \varphi + s \dfrac{ \varphi_n}{x_n}=0, \quad \text{in $B_{1}^+,$}\\
\partial_{ x_n^{1-s}} \, \varphi =0 \quad \text{on $\{x_n=0\}$}.
\end{cases}
\end{equation}

First, we need a preliminary lemma.

\begin{lem}\label{H}
Assume that $$ 0<x_n^+ \le w \le (x_n + a)^+ \quad \mbox{in} \quad B_r (x_0),$$
for some $a >0$. Then 
$$\mbox{either} \quad (x_n + c a)^+ \le w \quad \mbox{or} \quad w \le (x_n + (1-c)a)^+ \quad \mbox{in} \quad  B_{r/2} (x_0),$$
for some constant $c>0$ independent of $\gamma$.
\end{lem}

\begin{proof}It is convenient to prove this result working directly with the solution $u$ prior to applying the change of variables \eqref{change}. Then, our assumption reads:
 $$ 0<u_0(x_n) \le u \le u_0(x_n + a) \quad \mbox{in} \quad B_r (x_0).$$
After a dilation we may assume that $x_0=e_n$ and $r \le 1/2$. Assume for simplicity that 
$$u(e_n) \ge u_0(1+ \frac a 2).$$ Then $g:=u-u_0 \ge 0 $ satisfies
$$|\Delta g| \le C g \quad \mbox{in} \quad B_r(e_n).$$
and by Harnack inequality we find $$g \ge c g(e_n) \ge c(u_0(1+ \frac a2) - u_0(1))\quad \mbox{in} \quad B_{r/2}(e_n).$$ 
The conclusion follows since
$$c(u_0(1+ \frac a2) - u_0(1)) \ge u_0(x_n + c'a) - u_0(x_n) \quad \quad \forall x_n \in \left[\frac 12, \frac 32 \right],$$
provided that $c'$ is chosen sufficiently small. 
\end{proof}

\

\noindent \textit{Proof of Proposition $\ref{P1}$.}
Assume that for a sequence of $\eps_k \to 0$, $s_k \in (-1,0)$ and 
\begin{equation}\label{2001}
w_k \in \mathcal S(1,\eps_k)
\end{equation} which satisfy \eqref{weq}-\eqref{2000} the conclusion does not hold for some $\rho$ small, depending only on $n$, to be made precise later. 

By passing to a subsequence we may assume that
$$ s_k \to \bar s \in [-1,0].$$
{\it Claim 1:} Up to a subsequence, the graphs of
$$ \tilde w_k := \frac{w_k-x_n}{\eps_k} $$
defined in $ \overline{\{w_k >0\}}$, converge uniformly on compact sets to the graph of a H\"older limiting function $\bar w$  defined in $B_1 \cap \{x_n \ge 0\}$, and $$|\bar w|\le 1 \quad \bar w(0)=0.$$ 
Towards this, we notice that \eqref{2001} implies that the oscillation of $w_k$ decays geometrically in dyadic balls of radius $r_m=2^{-m}$ which are centered on the free boundary, $m \ge 1$. Indeed, if for example we focus at the origin, the unit directions $\nu^k_m$ of the linear functions which approximate $w_k$ in the balls $B_{r_m}$ satisfy
$$|\nu^k_{m}-\nu^k_{m-1}| \le C \eps \quad  \Longrightarrow \quad |\nu^k_m - e_n| \le C m \eps_k.$$
Then the inequalities
$$(x \cdot \nu^k_m - \eps_k r_m)^+ \le w_k \le (x \cdot \nu^k_m + \eps_k r_m)^+ \quad \quad \mbox{in} \quad B_{r_m},$$
 imply
 $$(x_n -(Cm +1) \eps_k r_m)^+ \le w_k \le (x_n + (Cm+1)\eps_k r_m)^+ \quad \quad \mbox{in} \quad B_{r_m},$$
 which gives
 $$ osc \, \, \tilde w_k \le C' m r_m \quad \quad \mbox{in} \quad B_{r_m} \cap \overline {\{w_k >0\}}.$$
 On the other hand, for dyadic balls included in $\{w_k>0\}$, we can apply Lemma \ref{H} and conclude that the oscillation of $\tilde w_k$ decays 
 geometrically as well. By combining these estimates we find that $\tilde w_k$ has a uniform H\"older modulus of continuity when restricted to compact sets of $B_1$, and the claim follows from Arzela-Ascoli theorem.  
 
 \
 
 {\it Claim 2:} $\bar w$ solves \eqref{LiE} in the viscosity sense, with $s=\bar s$. 
 
 If $\bar s=-1$ then the boundary condition is understood as in \eqref{1005}.
 
 \

The function $\tilde w_k$ solves the equation
$$ \Delta w_k = \frac {1}{ \eps_k} \frac{s_k \cdot h(e_n+ \eps_k \nabla \tilde w_k)}{x_n + \eps_k \tilde w_k }:= g(\eps_k,s_k,x_n, \tilde w_k,\nabla \tilde w_k)$$
and 
$$g(\eps_k,s_k,x_n, z, p) \to  \bar s \, \frac {\nabla h(e_n)  \cdot p}{x_n}=- \bar s \frac{p_n}{x_n} \quad \mbox{as $k \to \infty$.}$$
This shows that $\bar w$ solves the linear equation
\begin{equation}\label{501}
\Delta \bar w + \bar s \frac{\bar w_n}{x_n} =0 \quad \mbox{in} \quad B_1^+,
\end{equation}
in the viscosity sense. 

 It remains to show that $\bar w$ satisfies the boundary condition of \eqref{LiE} on $\{x_n=0\}$. 
 
 Towards this aim we construct explicit barriers. Let $p(x')$ be a given quadratic polynomial and let $d$ denote the distance to the graph $$x_n = -\eps p(x').$$
 We restrict our computations to the region
$$ B_1 \cap \{x_n \ge - \eps p(x') \},$$
and we have
 $$d=x_n + \eps p(x') + O(\eps^2),$$ 
 $$ \Delta d = \eps \kappa + O(\eps^2), \quad \quad \kappa=\Delta_{x'} p.$$
We let
\begin{equation}\label{2002}
\Phi:= d + \eps f(d),
\end{equation} 
for some one-dimensional Lipschitz function $f$, $f(0)=0$ to be specified below, and compute
\begin{align*}
\Delta \Phi  &= \eps f''(d)+(1+ \eps f'(d))\Delta d \\
&= \eps \left ( f''(d)+ \kappa + O(\eps)     \right),
\end{align*}
and
\begin{align*}
\frac{h_s(\nabla \Phi)}{\Phi} &=\frac s 2 \cdot \frac{-2\eps f'(d)+\eps^2 (f'(d)^2)}{d+\eps f(d)}\\
&= - \eps \, s \, \frac{f'(d)}{d} \left(1+ \eps O(|f'| + |\frac f d|)\right).
\end{align*}
The corresponding function $\tilde \Phi$ (see \eqref{tidw}) has the form 
\begin{equation}\label{2003}
\tilde \Phi = p(x') + f(x_n)+ O(\eps).
\end{equation}

We distinguish 2 cases, $\bar s \in (-1,0]$ and $\bar s=-1$.

\smallskip

{\it Case 1: $\bar s \in (-1,0]$.}

Assume by contradiction that, say for simplicity, $\bar w$ is touched strictly by below at $0$ by 
$$q(x):= - a |x'|^2 + t x^{1-\bar s}_n  ,$$
for some constants $a$, $t>0$. Then we pick $p(x')=- \frac a 2 |x'|^2$ and in \eqref{2002} we make the choice
$$f_k(d)= \frac t 4 d^{1-s_k} + \frac t 4  d^{1-2s_k} + M d^2,$$
for some large $M>0$. 

Then $\Phi_k$ is a strict viscosity subsolution since the free boundary condition in Definition \ref{defnw} is clearly satisfied by the choice of $f_k$, and the computations above imply that
$$
\Delta \Phi_k - \frac{h_{s_k}(\nabla \Phi_k)}{\Phi_k} = $$
$$ =\eps_k \left( \frac t 4|s_k| (1-2 s_k + O(\eps_k)) d^{-1-2 s_k}  + 2M (1+s_k) - (n-1)a + O(\eps_k)\right) >0,$$
for all large $k$'s.  

From \eqref{2003} we get the uniform convergence 
$$\tilde \Phi_k \to - \frac a 2 |x'|^2 + \frac t 4 d^{1-\bar s} + \frac t 4  d^{1-2 \bar s} + M d^2,$$
and the limit function, in a small neighborhood of $0$, touches $q$ and therefore $\bar w$ strictly by below at the origin.
This means that a translation of the graph of $\Phi_k$ in $\overline {\{\Phi_k>0\}}$ touches the graph of $w_k$ in $\overline {\{w_k>0\}}$ at an interior point and we reached a contradiction.

\smallskip

{\it Case 2: $\bar s=-1$.} 

Assume that we can touch $\bar w$ on $\{x_n=0\}$ by a quadratic polynomial $p(x')$ strictly by below at $0$, with a quadratic polynomial $p(x')$ with $$\Delta_{x'} p=t>0.$$ Then, we can touch $\bar w$ strictly by below at $0$ in $\overline{B_\delta^+}$ with
$$ q(x)= p(x') + \frac t 4 x_n^2 |\log x_n| - M x_n^2,$$
for some $M$ sufficiently large. This follows from the fact that $q$ is a subsolution of the equation \eqref{501} with $\bar s=-1$. 

We choose 
$$f_k(d)= \frac t 4 \cdot \frac{1}{1+s_k} (d^{1-s_k} - d^2) - M d^2   $$
and the corresponding function $\Phi_k$ is a viscosity subsolution since
$$
\Delta \Phi_k - \frac{h_{s_k}(\nabla \Phi_k)}{\Phi_k} = $$
$$ =\eps_k \left( - \frac t 2 - 2M (1+s_k) + t + O(\eps_k)d^{-1-2s_k}\right) >0,$$
for all large $k$'s. Notice that $\tilde \Phi_k \to q$ uniformly in $B_1^+$ and we reach a contradiction as in Step 1. With this {\it Claim 2} is proved.

\

Next we apply Proposition \ref{UE} and Corollary \ref{C09} to $\bar w$ and conclude that
$$|\bar w-a' \cdot x'| \le \frac{\rho}{8} \quad \mbox{in} \quad \overline {B}_\rho^+,$$
for some $\rho>0$ universal depending only on $n$. This implies
$$\left(x_n + \eps_k (a' \cdot x' - \frac \rho 4)\right)^+ \le w_k \le \left(x_n + \eps_k (a' \cdot x' + \frac \rho  4)\right)^+  \quad \mbox{in} \quad \overline B_\rho^+,$$
holds for large $k$'s. Then the conclusion is satisfied for $w_k$ with $\nu_k = \frac{e_n + \eps_k a'}{|e_n + \eps_k a'|}$ and we reached a contradiction.
\qed

\section{Proof of Theorem \ref{TM}}

In this final section we provide the proof of our main result, Theorem \ref{TM}. For the case when $\gamma \to 2$, it will follow from the characterization of global minimizers of the Dirichlet-perimeter functional $\mathcal F$ introduced in Section 2, combined with the compactness Theorem \ref{TMDS3}, and the uniform improvement of flatness result in the previous section. 

Here, we need to characterize the global minimizers of $\mathcal{F}$ and deduce the flatness of global minimizers of $\mathcal{E_\gamma}$.
We start with an important tool, that is a monotonicity formula for the Dirichlet-perimeter functional $\mathcal F$.

\begin{prop}[Monotonicity Formula]\label{WMF}
Let $(u,E)$ be a minimizing pair for $\mathcal F$ in $\Omega$. Then
$$ \Phi(r)=r^{1-n} \mathcal F_{B_r}(u,E) - \frac 12 r^{-n} \int_{\p B_r} u^2 d \sigma$$
is monotone increasing in $r$ as long as $B_r \subset \Omega$. 

Moreover, $\Phi$ is constant if and only if $(u,E)$ is a cone i.e $u$ is homogenous of degree $1/2$ and $E$ is homogenous of degree $0$.

\end{prop}

\begin{proof}
We compute for a.e. $r$

\begin{align*}
 \Phi'(r) & =  r^{1-n} \left(\int_{\p B_r} |\nabla_\theta u|^2  + \frac 14 \frac {u^2}{r^2}  d \sigma - \frac{n-1}{r} \mathcal F_{B_r}(u,E) \right. \\
 & + \left.  \int_{\p B_r} (u_\nu - \frac 12 \frac ur)^2 d\sigma  + \int_{\p E \cap \p B_r} (\sin \theta)^{-1} d \mathcal H^{n-2}  \right)
 \end{align*}
 where $\theta$ is the angle between the normal $\nu$ to $\p E$ and the radial direction $x/|x|$.

 Let $(\tilde u, \tilde E)$ be the extension of the boundary data of $(u,E)$ on $\p B_r$ to $\R^n$, with $\tilde u$ homogenous of degree $1/2$, and $\tilde E$ homogeneous of degree $0$. 

Denote by $\tilde \Phi$ the corresponding expression for the pair $(\tilde u, \tilde E)$. The homogeneity of the pair implies that $\tilde \Phi$ is constant in its argument and the computation above shows that
 $$\tilde \Phi'(r)= r^{1-n} \left(\int_{\p B_r} |\nabla_\theta \tilde u|^2  + \frac 14 \frac {\tilde u^2}{r^2}  d \sigma - \frac{n-1}{r} \mathcal F_{B_r}(\tilde u,\tilde E) + \mathcal H^{n-2} (\p \tilde E \cap \p B_r) \right).$$
 Now the conclusion $$\Phi'(r) \ge \tilde \Phi'(r)=0$$ follows since $(u,E)$ and $(\tilde u, \tilde E)$ coincide on $\p B_r$ and $$  \mathcal F_{B_r}(\tilde u,\tilde E)\ge \mathcal F_{B_r}(u,E)$$ by minimality of $(u,E)$. 
\end{proof}

We can now easily deduce the following result.

\begin{prop}\label{P11}
If $(u,E)$ is a cone, then $u \equiv 0$ and $E$ is a minimizing cone for the perimeter.
\end{prop}

\begin{proof}
In Theorem 4.1 of [ACKS] it was shown that $u$ is Lipschitz in the interior of the domain for a minimizing pair $(u,E)$. Since $u$ is homogenous of degree 1/2, this means that $ u =0$, and $E$ is a minimal cone for the perimeter.
\end{proof}

Next, we can characterize global minimizers to $\mathcal F$ in low dimensions, on the basis of the classical regularity theory for minimal surfaces (see for example \cite{G}).

\begin{prop}\label{P1'}
Assume $n \le 7$ and $(u,E)$ is a global minimizer for $\mathcal F$ with $0 \in \p E$. Then $u \equiv 0$ and $E$ is a half-space.
\end{prop}

\begin{proof}
If $n \le 7$ then, by Proposition \ref{P11} and Simons theorem for minimal surfaces, there is only one cone up to rotations i.e. $u \equiv 0$ and $E$ is a half-space. This means that the tangent cone at infinity and the tangent cone at $0$ for the pair $(u,E)$ have the same $\Phi$ value. This means that $\Phi$ is constant and $(u,E)$ is a cone.
\end{proof}

We deduce the following flatness property for the free boundaries of global minimizers of $J_\gamma$, as $\gamma \to 2.$

\begin{prop}\label{P2}
Assume $n \le 7$. Given $\eps>0$, there exist $R$ large and $\delta >0$ small depending on $\eps$ and $n$, such that if $u$ is a minimizer of $J_\gamma$ in $B_R$, and $0 \in \p \{u>0\}$, $\gamma \ge 2-\delta,$ then, up to rotations,
$$\{x_n \ge  \eps\} \cap B_1 \subset  \{u>0 \} \cap B_1 \subset \{x_n \ge - \eps\}.$$
\end{prop}

\begin{proof}
This follows easily by compactness from Theorem \ref{TMDS3}. Indeed, for a sequence of $\gamma_m \to 2^-$ and minimizers $u_m$ defined in $B_m$ for $\mathcal E_{\gamma_m}$, we can extract a subsequence so that the free boundaries $\p \{u_m>0\}$ converges uniformly on compact sets to $\p E$ for some global minimizing pair $(u,E)$ of $\mathcal F$. Proposition \ref{P1} implies that $u_m$ satisfies the conclusion for all large $m$'s.
\end{proof}

Finally, we also  obtain the flatness of global minimizers of $\mathcal E_\gamma$, as $\gamma \to 2$.

\begin{lem}\label{LF}
Assume $n \le 7$, and that $u$ is a global minimizer of $\mathcal E_\gamma$, $\gamma \ge 2 - \delta$, with $\delta$ as in Proposition $\ref{P2}$. Then,
$$(1-C \eps) u_0(d(x)) \le u(x) \le (1 + C \eps) u_0(d(x)),$$
with $C$ depending only on $n$, and $d(x):= dist(x, F(u))$.
\end{lem}

Proposition \ref{P2} and Lemma \ref{LF} imply that global minimizers $u$ of $\mathcal E_\gamma$, for $\gamma \to 2$, satisfy the flatness assumption $u \in \mathcal S(r,C\eps)$ for all $r$'s. Then the conclusion of Theorem \ref{TM} follows by applying Theorem \ref{TM} indefinitely. We are left with the proof of Lemma \ref{LF}.

Before proving Lemma \ref{LF}, we need the following preliminary result. Notice that, the multiples of the one dimensional solution 
$$ a u_0(x_n), \quad a >0,$$
are supersolutions for the Euler-Lagrange equation when $ a  \ge 1$ and subsolutions when $a \le 1$.
Next we show that if a solution $u$
is trapped between two such multiples, then the bounds can be improved in a linear fashion in the interior.
\begin{lem}\label{L010}
Assume that $\gamma \in [1,2)$ and $u$ satisfies
$$
\Delta u = W'(u) \quad \mbox{in} \quad B_{1/2}(e_n),
$$
and
$$a_-  \le \frac{ u}{u_0(x_n)} \le a_+, $$ for some 
$$0< a_- \le 1 \le a_+.$$ 
Then $$(1-c)a_- +c \le \frac{u}{u_0(x_n)} \le (1-c)a_+ +c \quad \mbox{in} \quad B_{1/4}(e_n),$$
for some constant $c>0$ depending only on $n$.

\end{lem}

We require for $\gamma$ to be bounded away from 0 in order to have an inequality $$|W'(u_0)| \ge c \quad \mbox{ in} \quad  B_{1/2}(e_n),$$ with $c$ universal.

\begin{proof}
Let $v:=a \, u_0(x_n)$ and notice that in $B_{1/2}(e_n)$
\begin{align*}
\Delta v & = a W'(u_0)=a^{\gamma +2} W'(v) = W'(v) + \frac{a^{\gamma+2}-1}{a^{\gamma+1}}W'(u_0) \\
& \le W'(v) - c(a-1),
\end{align*}
with $c$ independent of $\gamma$.
Then, using the uniform Lipschitz bound of $W'$ we find that $w=v-u \ge 0$ satisfies
$$ \Delta w \le  C w - c(a-1) \quad \mbox{in} \quad B_{1/2}(e_n).$$
Now we can use comparison with explicit quadratic polynomials of the type
$$ \mu(1-C|x-x_0|^2) , \quad x_0 \in B_{1/4}(e_n)$$
and obtain
$$w \ge c'(a-1) \ge c'' (a-1)u_0 \quad \mbox{in} \quad B_{1/4}(e_n) ,$$
which gives the upper bound. 

The lower bound follows in a similar fashion.
\end{proof}

\

\noindent \textit{Proof of Lemma $\ref{LF}$.}
We show that 
\begin{equation}\label{am}
u(x) \le a_m u_0(d(x))
\end{equation} for successive constants $a_m$ that decrease to $1+ C '\eps$. 

In \cite{DS1} we showed that \eqref{am} holds for some $a_0$ large depending on $\gamma$. Suppose that \eqref{am} is satisfied for some constant $a_m$, and since the statement remains invariant under the rescaling of the equation, we may assume that $B_1(e_n) \subset \{u>0\}$ is tangent to the free boundary at $0$. By Proposition \ref{P2} we know that $\p \{u>0\} \cap B_4$ is trapped in the strip $\{|x_n| \le 4 \eps\}$. Then, \eqref{am} gives $$ u(x) \le a_m u_0(x_n + 4 \eps) \le a_m (1+C \eps) u_0(x_n) \quad \mbox{in} \quad B_{1/2}(e_n),$$
with $C$ independent of $\gamma$. We apply Lemma \ref{L010} to obtain $$u(e_n) \le a_{m+1} u_0(e_n), \quad a_{m+1}:=a_m(1+C\eps) (1-c) +c,$$
and, after rescaling, we find that the constant $a_m$ can be replaced by $a_{m+1}$, and the claim easily follows. 

The lower bound can be proved in a similar way.
\qed

\

We are now left with the proof of Theorem \ref{TM}, in the case $\gamma \to 0$. The only missing ingredient is the following compactness result 
showing that the limit of minimizers of $\mathcal E _\gamma$ with exponents tending to 0 is a minimizer for $\mathcal E_0$ as well. Recall that $\mathcal E_0$ is the Alt-Caffarelli functional:
$$\mathcal E_0 (u):= \int_\Omega (|\nabla u|^2+ \chi_{\{u>0\}})\; dx,$$ for which regularity in low dimension was established in \cite{AC, CJK, JS}.

\begin{prop}[Compactness for $\gamma \to 0$]\label{PCM}
Assume that $$\gamma_k \to 0, \quad \quad u_k \to \bar u \quad \mbox{in} \quad L^2(\Omega),$$ and $u_k$ are minimizers of $\mathcal E_{\gamma_k}$. Then $\bar u$ is a minimizer for $ \mathcal E_{0}$ in $\Omega$. 
\end{prop}

A version of this proposition for a fixed exponent $\gamma$ was proved in \cite{DS1}. It relies on a construction which interpolates between two functions which are $L^2$ close in an annulus, without increasing too much the $\mathcal E_\gamma$ energy.
\begin{lem}\label{uv} Let $u_k,v_k$ be sequences in $H^1(B_1)$ and $\delta>0$ small. Assume that $u_k-v_k \to 0$ in $L^2(B_{1-\delta/2} \setminus \bar B_{1-\delta}),$ as $k\to \infty,$ and that $u_k, v_k$ have uniformly (in $k$) bounded energy in $B_{1-\delta/2}$. Then, there exists $w_k \in H^1(B_1)$ with $$w_k:= \begin{cases}v_k \quad \text{in $B_{1-\delta}$}\\ u_k \quad \text{in $B_1 \setminus \bar B_{1-\delta/2}$}\end{cases}$$ such that 
 $$\mathcal E_\gamma (w_k, B_1) \leq \mathcal E_\gamma (u_k, B_{1-\delta/2}) + \mathcal E_\gamma (v_k, B_1 \setminus \bar B_{1-\delta}) + o(1),$$
 with $o(1) \to 0$ as $k \to \infty.$
\end{lem}

An inspection of the proof in \cite{DS1} shows that the result is valid for a sequence of exponents $\gamma_k$ that remain bounded away from 2. The reason is that the dependence of the constants on $\gamma$ is uniform as long as $\gamma$ is restricted to a compact set of $[0,2)$.
  
   We sketch the proof of Proposition \ref{PCM} in this more general setting of variable exponents.
    
\begin{proof}[Proof of Proposition $\ref{PCM}$]
Assume for simplicity that $\Omega=B_1$, and let $\bar v$ be a competitor for $\bar u$ in $B_1$ with $\bar v = \bar u$ in $B_1 \setminus \bar B_{1-\delta}$ for $\delta>0$ small, and with
 $$\mathcal E_{0}(\bar v, B_1) + \|v\|_{L^\infty(B_1)}< \infty.$$ 
First we construct a sequence of truncations 
$$v_k:=(\bar v-t_k)^+ \quad \mbox{ with $t_k \to 0$}$$ 
such that
\begin{equation}\label{vkb1}
\mathcal E_{\gamma_k} (v_k,B_1) \to \mathcal E_{0}(\bar v,B_1).
\end{equation}
Notice that for any $\eta>0$, there exists $t \in [0,\eta]$ such that
\begin{equation}\label{barv}
 \int_{B_1} ((\bar v-t)^+) ^{-\frac 12}< \infty,
 \end{equation}
which follows from 
$$\int_0^\eta \int_{B_1}[(\bar v-t)^+]^{-\frac 12} dx dt  \le C \eta^{1/2}.$$  
By Lebesgue dominated convergence theorem, \eqref{barv} implies
$$ \mathcal E_{\gamma} ((\bar v-t)^+,B_1) \to \mathcal E_{0} ((\bar v-t)^+,B_1)  \quad \mbox{as $\gamma \to 0$,} $$
and the claim \eqref{vkb1} follows since
$$\mathcal E_{0} ((\bar v-t)^+,B_1) \to \mathcal E_0(\bar v,B_1)  \quad \mbox{as $t \to 0$.}$$ 
The lower semicontinuity property for subdomains $D \subset B_1$,
$$\ \liminf \mathcal E_{\gamma_k} (v_k,D) \ge \mathcal E_{0}(\bar v,D),$$
implies that the convergence in \eqref{vkb1} holDS1 also for subdomains $D \subset B_1$.
 
 We use Lemma \ref{uv}, and call $w_k$ the interpolation of $v_k$ and $u_k$ such that $$w_k= v_k \quad \text{in $B_{1-\delta,}$} \quad w_k=u_k \quad \text{in $B_1 \setminus \bar B_{1-\delta/2}$}.$$ The hypotheses of Lemma \ref{uv} apply since 
 $$u_k-v_k \to 0 \quad\mbox{ in}  \quad L^2(B_1 \setminus B_{1-\delta}),$$ 
 and $\mathcal E_{\gamma_k}(u_k,B_{1-\delta/2})$ is uniformly bounded by Lemma 3.4 in \cite{DS1}.
 Then, by the minimality of $u_k$ and Lemma \ref{uv}, we get
$$\mathcal E_{\gamma_k}(u_k, B_1) \leq \mathcal E_{\gamma_k}(w_k, B_1) \leq \mathcal E_{\gamma_k}(v_k, B_{1-\delta/2}) + \mathcal E_{\gamma_k}(u_k, B_1 \setminus \bar B_{1-\delta}) + o(1),$$ with $o(1) \to 0$ as $k\to \infty.$ 
Subtract $\mathcal E_{\gamma_k}(u_k, B_1 \setminus \bar B_{1-\delta})$ from both sides, and obtain
$$\mathcal E_{\gamma_k}(u_k, B_{1-\delta}) \leq  \mathcal E_{\gamma_k}(v_k,  B_{1-\delta/2}) + o(1).$$
The lower semi-continuity of $\mathcal E$, and the convergence of the energies for the $v_k$'s gives
$$\mathcal E_0(\bar u, B_{1-\delta}) \leq  \mathcal E_0( \bar v, B_{1-\delta/2}).$$ We obtain the conclusion by letting $\delta \to 0.$
\end{proof}

\end{document}